\documentclass[11pt]{article}%
\usepackage{amsmath}
\usepackage{amsfonts}
\usepackage{amssymb}
\usepackage{graphicx}%
\newtheorem{theorem}{Theorem}

\newtheorem{corollary}[theorem]{Corollary}

\newtheorem{lemma}[theorem]{Lemma}

\newenvironment{proof}[1][Proof]{\noindent\textbf{#1.} }{\ \rule{0.5em}{0.5em}}
\begin{document}

\centerline{{\Large Reduction of Order, Periodicity and Boundedness in }}
\vspace{0.5ex} \centerline{{\Large Nonlinear, Higher Order Difference Equations }}

\vspace{2ex}

\centerline{H. SEDAGHAT \footnote{\noindent Department of Mathematics, Virginia Commonwealth University, Richmond, Virginia, 23284-2014, USA; Email: hsedagha@vcu.edu}}

\vspace{2ex}

\begin{abstract}
\noindent We consider the semiconjugate factorization and
reduction of order for non-autonomous, nonlinear, higher order difference equations 
containing linear arguments. These equations have appeared in several
mathematical models in biology and economics. By extending some recent results 
to cases where characteristic polynomials of the linear
expressions have complex roots, we obtain new results on boundedness and the
existence of periodic solutions for equations of order 3 or greater.

\end{abstract}

\bigskip

\section{Introduction}

Special cases of the following type of higher order difference equation have
frequently appeared in the literature:%
\begin{equation}
x_{n+1}=\sum_{i=0}^{k}a_{i}x_{n-i}+g_{n}\left(  \sum_{i=0}^{k}b_{i}%
x_{n-i}\right)  ,\quad n=0,1,2,\ldots\label{gla}%
\end{equation}

We assume here that $k$ is a fixed positive integer and for each $n$, the
function $g_{n}:\mathbb{R}\rightarrow\mathbb{R}$ is defined on the real line.
The parameters $a_{i},b_{i}$ are fixed real numbers such that
\[
a_{k}\not =0\text{ or }b_{k}\not =0.
\]

Upon iteration, Equation (\ref{gla}) generates a unique sequence of points
$\{x_{n}\}$ in $\mathbb{R}$ (its solution) from any given set of
$k+1$ initial values $x_{0},x_{-1},\ldots,x_{-k}\in\mathbb{R}$. The
number $k+1$ is the order of (\ref{gla}).

Special cases of Equation (\ref{gla}) appeared in the classical economic
models of the business cycle in twentieth century in the works of Hicks
\cite{Hic}, Puu \cite{Puu}, Samuelson \cite{Sam} and others; see \cite{bk1},
Section 5.1 for some background and references. Other special cases of
(\ref{gla}) occurred later in mathematical studies of biological models
ranging from whale populations to neuron activity; see, e.g., Clark
\cite{Clr}, Fisher and Goh \cite{FG}, Hamaya \cite{Ham} and Section 2.5 in
Kocic and Ladas \cite{KL}.

The dynamics of special cases of (\ref{gla}) have been investigated by several
authors. Hamaya uses Liapunov and semicycle methods in \cite{Ham} to obtain
sufficient conditions for the global attractivity of the origin for the
following special case of (\ref{gla})
\[
x_{n+1}=\alpha x_{n}+a\tanh\left(  x_{n}-\sum_{i=1}^{k}b_{i}x_{n-i}\right)
\]
with $0\leq\alpha<1$, $a>0$ and $b_{i}\geq0$. These results can also be
obtained using only the contraction method in \cite{Mem} and \cite{Sgeo}; also see 
\cite{Liz} for a discussion of alternative methods. The results in
\cite{Sgeo} are used in \cite{bk1}, Section 4.3D, to prove the global
asymptotic stability of the origin for an autonomous special case of
(\ref{gla}) with $a_{i},b_{i}\geq0$ for all $i$ and $g_{n}=g$ for all $n$,
where $g$ is a continuous, non-negative function. The study of global
attractivity and stability of fixed points for other special cases of
(\ref{gla}) appear in \cite{GLV} and \cite{KPS}; also see \cite{KL}, Section 6.9.

The second-order case ($k=1$) has been studied in greater depth. Kent and
Sedaghat obtain sufficient conditions in \cite{KS} for the boundedness and
global asymptotic stability of
\begin{equation}
x_{n+1}=cx_{n}+g(x_{n}-x_{n-1}) \label{sed}%
\end{equation}

In \cite{Elm}, El-Morshedy improves the convergence
results of \cite{KS} for (\ref{sed}) and also gives necessary and sufficient
conditions for the occurrence of oscillations. The boundedness of solutions of
(\ref{sed}) is studied in \cite{S97} and periodic and monotone solutions of
(\ref{sed}) are discussed in \cite{Sed2}. Li and Zhang study the bifurcations
of solutions of (\ref{sed}) in \cite{LZ}; their results include the
Neimark-Sacker bifurcation (discrete analog of Hopf).

A more general form of (\ref{sed}), i.e., the following equation
\begin{equation}
x_{n+1}=ax_{n}+bx_{n-1}+g_{n}(x_{n}-cx_{n-1}) \label{sed1}%
\end{equation}
is studied in \cite{SKy} where sufficient conditions for the occurrence of
periodic solutions, limit cycles and chaotic behavior are obtained using
reduction of order and factorization of the above difference equation into a
pair of equations of lower order. These methods are used in \cite{dkmos} to
determine sufficient conditions on parameters for occurrence of limit cycles
and chaos in those rational difference equations of the following type%
\begin{equation}
x_{n+1}=\frac{ax_{n}^{2}+bx_{n-1}^{2}+cx_{n}x_{n-1}+dx_{n}+ex_{n-1}+f}{\alpha
x_{n}+\beta x_{n-1}+\gamma} \label{ro2}%
\end{equation}
that can be reduced to special cases of (\ref{sed1}).

In this paper, we consider the possible occurrence of complex roots for the
characteristic polynomials associated with the linear expressions $\sum
_{i=0}^{k}a_{i}u_{i}$ and $\sum_{i=0}^{k}b_{i}u_{i}$ in (\ref{gla}). Complex,
non-real roots that are common to both polynomials may occur when (\ref{gla})
has order 3 or greater ($k\geq2$), a situation that cannot occur in the
second-order equations of \cite{dkmos} or \cite{SKy}. We show that the
resulting decomposition of (\ref{gla}) into lower-order equations via complex
conjugate roots nevertheless yields a factor-cofactor pair in the sense of
\cite{arx} or \cite{bk2} with real coefficients and a type-$(k-1,2)$ reduction
of (\ref{gla}) within the real number system. We use this factorization to
study boundedness and the occurrence of periodic solutions and limit cycles
for (\ref{gla}). The results obtained here extend similar results in
\cite{SKy} to equations of order 3 and greater for the first time.

\section{Reduction of order}

We begin with a result from \cite{bk2} (Theorem 5.6) that we quote here as a
lemma. A generalization of this lemma to algebras over fields is proved in
essentially the same way; see \cite{arx2}.

\begin{lemma}
\label{fsor}Let $g_{n}:\mathcal{F}\rightarrow\mathcal{F}$ be a sequence of
functions on a field $\mathcal{F}$. If for $a_{i},b_{i}\in\mathcal{F}$ the
polynomials%
\[
P(u)=u^{k+1}-\sum_{i=0}^{k}a_{i}u^{k-i},\quad Q(u)=\sum_{i=0}^{k}b_{i}u^{k-i}%
\]
have a common, nonzero root $\rho\in\mathcal{F}$ then each solution
$\{x_{n}\}$ of (\ref{gla}) in $\mathcal{F}$ satisfies%
\begin{equation}
x_{n+1}=\rho x_{n}+t_{n+1} \label{cfe}%
\end{equation}
where the sequence $\{t_{n}\}$ is the unique solution of the equation:%
\begin{equation}
t_{n+1}=-\sum_{i=0}^{k-1}p_{i}t_{n-i}+g_{n}\left(  \sum_{i=0}^{k-1}%
q_{i}t_{n-i}\right)  \label{fe}%
\end{equation}
in $\mathcal{F}$ with initial values $t_{-i}=x_{-i}-\rho x_{-i-1}$ for
$i=0,1,\ldots,k-1$ and coefficients%
\[
p_{i}=\rho^{i+1}-a_{0}\rho^{i}-\cdots-a_{i}\quad\text{and\quad}q_{i}=b_{0}%
\rho^{i}+b_{1}\rho^{i-1}+\cdots+b_{i}%
\]
in $\mathcal{F}$. Conversely, if $\{t_{n}\}$ is a solution of (\ref{fe}) with
initial values $t_{-i}\in\mathcal{F}$ then the sequence $\{x_{n}\}$ that it
generates in $\mathcal{F}$ via (\ref{cfe}) is a solution of (\ref{gla}).
\end{lemma}

The preceding result shows that Equation (\ref{gla}) splits into the
equivalent pair of equations (\ref{cfe}) and (\ref{fe}) provided that the
polynomials $P$ and $Q$ have a common nonzero root $\rho$. We call the pair of
equations (\ref{cfe}) and (\ref{fe}) a \textit{semiconjugate factorization} of
(\ref{gla}). Equation (\ref{fe}), whose order is one less than the order of
(\ref{gla}) is the\textit{ factor equation} and Equation (\ref{cfe}) which
bridges the order (or dimension) gap between (\ref{gla}) and (\ref{fe}) is the
\textit{cofactor equation}.

Since Equation (\ref{fe}) is of the same type as (\ref{gla}) we may consider
applying Lemma \ref{fsor}\ to (\ref{fe}) to obtain a further reduction of
order. This is done next.

\begin{lemma}
\label{fsor2}Let $k\geq2$ and assume that the coefficients of (\ref{gla}) are
complex, i.e., $a_{i},b_{i}\in\mathbb{C}$. Let $\mathcal{F}=\mathbb{C}$ in
Lemma \ref{fsor} and suppose that $g_{n}:\mathbb{C}\rightarrow\mathbb{C}$ are
complex functions for all $n$. If the polynomials $P,Q$ in Lemma \ref{fsor}
have two common, nonzero roots $\rho,\gamma\in\mathbb{C}$ then (\ref{fe}) has
a factor equation
\begin{equation}
r_{n+1}=-\sum_{j=0}^{k-2}p_{j}^{\prime}r_{n-j}+g_{n}\left(  \sum_{j=0}%
^{k-2}q_{j}^{\prime}r_{n-j}\right)  \label{fe1}%
\end{equation}
with coefficients%
\[
p_{j}^{\prime}=\gamma^{j+1}+p_{0}\gamma^{j}+\cdots+p_{j}\quad\text{and\quad
}q_{j}^{\prime}=q_{0}\gamma^{j}+q_{1}\gamma^{j-1}+\cdots+q_{j}.
\]
where the numbers $p_{j},q_{j}$ are as defined in Lemma \ref{fsor} in terms of
the root $\rho.$ There are two cofactor equations
\begin{align}
t_{n+1}  &  =\gamma t_{n}+r_{n+1}\label{cfe2a}\\
x_{n+1}  &  =\rho x_{n}+t_{n+1} \label{cfe2b}%
\end{align}
the second of which is just (\ref{cfe}) from Lemma \ref{fsor}. The triangular
system of three equations (\ref{fe1})-(\ref{cfe2b}) is equivalent to
(\ref{gla}) in the sense of Lemma \ref{fsor}; i.e., they generate the same set
of solutions.
\end{lemma}

\begin{proof}
Consider the polynomials associated with the factor equation (\ref{fe}), i.e.,%
\[
P_{1}(u)=u^{k}+\sum_{j=0}^{k-1}p_{j}u^{k-j-1},\quad Q_{1}(u )=\sum_{j=0}%
^{k-1}q_{j}u^{k-j-1}.
\]

Let $\rho$ be a root of $P.$ We claim that%
\[
(u-\rho)P_{1}(u)=P(u).
\]

This is established by straightforward calculation:%
\begin{align*}
(u-\rho)P_{1}(u)  &  =(u-\rho)\left(  u^{k}+\sum_{j=0}^{k-1}p_{j}%
u^{k-j-1}\right) \\
&  =u^{k+1}-\rho p_{k-1}+\sum_{j=0}^{k-1}(p_{j}-\rho p_{j-1})u^{k-j}%
\end{align*}

where we define $p_{-1}=1$ to simplify the notation. Using the definition of
the numbers $p_{i}$ in Lemma \ref{fsor} we obtain%
\[
p_{j}-\rho p_{j-1}=-a_{j}%
\]

and further, since $P(\rho)=0$ we obtain%
\begin{align*}
\rho p_{k-1}  &  =\rho(\rho^{k}-a_{0}\rho^{k-1}-\cdots-a_{k-1})\\
&  =P(\rho)+a_{k}\\
&  =a_{k}%
\end{align*}

which completes the proof of the claim. A similar argument shows that if
$\rho$ is a root of $Q$ then%
\[
(u-\rho)Q_{1}(u)=Q(u).
\]

Now, suppose that $\gamma$ is also a common root of $P$ and $Q.$ If
$\gamma\not =\rho$ then clearly $P_{1}(\gamma)=Q_{1}(\gamma)=0$ so $\gamma$ is
a common root of $P_{1}$ and $Q_{1}$. Otherwise, $\gamma=\rho$ and $\rho$ is a
double root, hence a zero of the derivatives $P^{\prime}$ and $Q^{\prime},$
i.e.,%
\[
P^{\prime}(\rho)=Q^{\prime}(\rho)=0.
\]

In addition, we find that%
\begin{align*}
P_{1}(\rho)  &  =\rho^{k}+\sum_{j=0}^{k-1}(\rho^{j+1}-a_{0}\rho^{j}%
-\cdots-a_{j-1}\rho-a_{j})\rho^{k-j-1}\\
&  =(k+1)\rho^{k}-\sum_{j=0}^{k-1}(k-j)a_{j}\rho^{k-j-1}\\
&  =P^{\prime}(\rho)
\end{align*}

so that $\rho=\gamma$ is a root of $P_{1}.$ Similarly, $\rho=\gamma$ is also
seen to be a root of $Q_{1}$. Now applying Lemma \ref{fsor} to (\ref{fe})
yields a factor equation (\ref{fe1}) and a cofactor (\ref{cfe2a}).

Finally, the last assertion follows from Theorem 3.1 in \cite{bk2} (or Theorem
6 in \cite{arx}).
\end{proof}

The next result on factorization of polynomials is also needed.

\begin{lemma}
\label{polyfac}Suppose that $\gamma,\rho\in\mathbb{C}$ are roots of the
polynomial $c_{0}u^{m}+c_{1}u^{m-1}+\cdots+c_{m-1}u+c_{m}$ of degree $m\geq2$
with coefficients $c_{j}\in\mathbb{R}$. Then
\begin{equation}
\sum_{j=0}^{m}c_{j}u^{m-j}=(u^{2}-(\gamma+\rho)u+\gamma\rho)\sum_{j=0}%
^{m-2}\alpha_{j}u^{m-j-2} \label{pf}%
\end{equation}
where $\alpha_{0}=c_{0}$,%
\begin{equation}
\alpha_{j}=c_{j}+(\gamma+\rho)\alpha_{j-1}-\gamma\rho\alpha_{j-2},\quad
j=1,2,\ldots,m-2,\ \alpha_{-1}\doteq0, \label{alfj}%
\end{equation}
and the following equalities hold:%
\begin{align}
c_{m-1}+(\gamma+\rho)\alpha_{m-2}-\gamma\rho\alpha_{m-3}  &  =0,\label{cm-1}\\
c_{m}-\gamma\rho\alpha_{m-2}  &  =0. \label{cm}%
\end{align}
Further, if $\gamma$ and $\rho$ are either both real or they are complex
conjugates then the numbers $\alpha_{j}$, $j=1,2,\ldots,m-2$ that satisfy the
recursions (\ref{alfj}) are real and found to be:%
\begin{align}
\alpha_{j}  &  =\sum_{i=0}^{j}\frac{\gamma^{i+1}-\rho^{i+1}}{\gamma-\rho
}c_{j-i},\quad\text{if }\gamma\not =\rho\label{alj1}\\
\alpha_{j}  &  =\sum_{i=0}^{j}(i+1)\rho^{i}c_{j-i},\quad\text{if }\gamma=\rho.
\label{alj2}%
\end{align}
Conversely, let $\sum_{j=0}^{m}c_{j}u^{m-j}$ be a polynomial with real
coefficients $c_{j}$. If $\gamma,\rho\in\mathbb{C}$ and there are real numbers
$\alpha_{j}$ satisfying (\ref{alfj})-(\ref{cm}) then (\ref{pf}) holds and
$\gamma,\rho$ are roots of $\sum_{j=0}^{m}c_{j}u^{m-j}$.
\end{lemma}

\begin{proof}
Assume that $\gamma,\rho\in\mathbb{C}$ are roots of $\sum_{j=0}^{m}%
c_{j}u^{m-j}.$ Then this polynomial is evenly divided by the quadratic
polynomial
\begin{equation}
(u-\gamma)(u-\rho)=u^{2}-(\gamma+\rho)u+\gamma\rho\label{q}%
\end{equation}

with a resulting quotient polynomial $\sum_{j=0}^{m-2}\alpha_{j}u^{m-j-2};$
i.e., (\ref{pf}) holds. To determine the coefficients $\alpha_{j}$ of the
quotient, multiply the polynomials on the right hand side of (\ref{pf}) and
rearrange terms to obtain the identity%
\begin{align*}
\sum_{j=0}^{m}c_{j}u^{m-j}  &  =\alpha_{0}u^{m}+(\alpha_{1}-(\gamma
+\rho)\alpha_{0})u^{m-1}+(\alpha_{2}-(\gamma+\rho)\alpha_{1}+\gamma\rho
\alpha_{0})u^{m-2}\\
&  \quad+\cdots+(\alpha_{m-2}-(\gamma+\rho)\alpha_{m-3}+\gamma\rho\alpha
_{m-4})u^{2}+\\
&  \qquad+(-(\gamma+\rho)\alpha_{m-2}+\gamma\rho\alpha_{m-3})u+\gamma
\rho\alpha_{m-2}.
\end{align*}

Now, matching coefficients on the two sides yields (\ref{alfj})-(\ref{cm}).

Next, if $\gamma$ and $\rho$ are either both real or they are complex
conjugates then $(\gamma+\rho)$ and $\gamma\rho$ are both real. In this case,
the numbers $\alpha_{j}$ defined by the recursions (\ref{alfj}) are also real.
Finally, (\ref{alj1}) and (\ref{alj2}) may be proved by induction. First,
suppose that $\gamma\not =\rho.$ For $j=1$ we have%
\[
c_{1}+\frac{\gamma^{2}-\rho^{2}}{\gamma-\rho}c_{0}=c_{1}+(\gamma+\rho
)\alpha_{0}=\alpha_{1}%
\]

so (\ref{alj1}) is true if $j=1.$ Suppose next that for $1\leq j\leq m-3,$
(\ref{alj1}) is true for $1,2,\ldots,j$. Then for $j+1$%
\begin{align*}
\alpha_{j+1}  &  =c_{j+1}+(\gamma+\rho)\alpha_{j}-\gamma\rho\alpha_{j-1}\\
&  =c_{j+1}+(\gamma+\rho)\sum_{i=0}^{j}\frac{\gamma^{i+1}-\rho^{i+1}}%
{\gamma-\rho}c_{j-i}-\gamma\rho\sum_{i=1}^{j}\frac{\gamma^{i}-\rho^{i}}%
{\gamma-\rho}c_{j-i}\\
&  =c_{j+1}+(\gamma+\rho)c_{j}+\sum_{i=1}^{j}\left[  (\gamma+\rho)\frac
{\gamma^{i+1}-\rho^{i+1}}{\gamma-\rho}-\gamma\rho\frac{\gamma^{i}-\rho^{i}%
}{\gamma-\rho}\right]  c_{j-i}.
\end{align*}

Since for each $i=1,\ldots,j$%
\[
(\gamma+\rho)(\gamma^{i+1}-\rho^{i+1})-\gamma\rho(\gamma^{i}-\rho^{i}%
)=\gamma^{i+2}-\rho^{i+2}%
\]

we obtain%
\[
\alpha_{j+1}=c_{j+1}+\frac{\gamma^{2}-\rho^{2}}{\gamma-\rho}c_{j}+\sum
_{i=1}^{j}\frac{\gamma^{i+2}-\rho^{i+2}}{\gamma-\rho}c_{j-i}%
\]

which verifies the induction step. If $\gamma=\rho$ then for $j=1$%
\[
c_{1}+2\rho c_{0}=c_{1}+2\rho\alpha_{0}=\alpha_{1}%
\]

so (\ref{alj2}) is true if $j=1.$ Suppose next that for $1\leq j\leq m-3,$
(\ref{alj2}) is true for $1,2,\ldots,j$. Then for $j+1$%
\begin{align*}
\alpha_{j+1}  &  =c_{j+1}+2\rho\alpha_{j}-\rho^{2}\alpha_{j-1}\\
&  =c_{j+1}+2\rho\sum_{i=0}^{j}(i+1)\rho^{i}c_{j-i}-\rho^{2}\sum_{i=1}%
^{j}i\rho^{i-1}c_{j-i}\\
&  =c_{j+1}+2\rho c_{j}+\sum_{i=1}^{j}[2(i+1)\rho^{i+1}-i\rho^{i+1}]c_{j-i}\\
&  =c_{j+1}+2\rho c_{j}+\sum_{i=1}^{j}(i+2)\rho^{i+1}c_{j-i}%
\end{align*}

which verifies the induction step.

Conversely, if $\gamma,\rho\in\mathbb{C}$ and $\alpha_{j}\in\mathbb{R}$
satisfy (\ref{alfj})-(\ref{cm}) then by the definition of $\alpha_{j}$ the
quadratic polynomial (\ref{q}) divides $\sum_{j=0}^{m}c_{j}u^{m-j}$ evenly.
Therefore, $\gamma,\rho$ are roots of $\sum_{j=0}^{m}c_{j}u^{m-j}$.
\end{proof}

\medskip

If the coefficients $a_{i},b_{i}$ in (\ref{gla}) are real and a common root
$\rho$ of $P$ and $Q$ is complex then these polynomials also share another
complex root, namely, the conjugate $\bar{\rho}$; thus, Lemma \ref{fsor2} is
applicable. However, if the functions $g_{n}:\mathbb{R}\rightarrow\mathbb{R}$
are real functions then a direct application of Lemma \ref{fsor2} is
problematic since the coefficients $p_{i},q_{i}$ of the factor equation
(\ref{fe}) are complex. The next result shows that this difficulty does not
actually arise since the coefficients $p_{i}^{\prime},q_{i}^{\prime}$ of the
secondary factor equation (\ref{fe1}) are in fact, real and furthermore, the
two complex cofactor equations in Lemma \ref{fsor2} combine into a single
second-order cofactor equation in $\mathbb{R}$.

\begin{theorem}
\label{cxrt}Let $k\geq2$ in (\ref{gla}) and assume that the coefficients
$a_{i},b_{i}$ are all real and $g_{n}:\mathbb{R}\rightarrow\mathbb{R}$ for
$n\geq0.$ If the polynomials $P,Q$ in Lemma \ref{fsor} have a common complex
root $\rho=\mu e^{i\theta}\not \in \mathbb{R}$ then the following statements
are true:

(a) The coefficients $p_{j}^{\prime},q_{j}^{\prime}$ of the factor equation
(\ref{fe1}) in Lemma \ref{fsor2} are real numbers that may be writtern in
terms of the original coefficients $a_{i},b_{i}$ of (\ref{gla}) as%
\begin{align}
p_{j}^{\prime}  &  =\mu^{j+1}\frac{\sin(j+2)\theta}{\sin\theta}-\frac{1}%
{\sin\theta}\sum_{m=0}^{j}a_{m}\,\mu^{j-m}\sin(j-m+1)\theta,\label{p1j}\\
q_{j}^{\prime}  &  =\frac{1}{\sin\theta}\sum_{m=0}^{j}b_{m}\,\mu^{j-m}%
\sin(j-m+1)\theta\label{q1j}%
\end{align}
for $j=0,1,\ldots,k-2.$

(b) The pair of first-order cofactor equations in Lemma \ref{fsor2} with
complex coefficients $\rho$ and $\gamma=\bar{\rho}$ combine into one
equivalent, second-order, non-homogeneous linear equation with real
coefficients%
\begin{equation}
x_{n+1}-2\mu\cos\theta\,x_{n}+\mu^{2}x_{n-1}=r_{n+1} \label{cfe2}%
\end{equation}
where the sequence $\{r_{n}\}$ is a solution of the factor equation
(\ref{fe1}) in $\mathbb{R}$.

(c) The system of equations (\ref{fe1}) and (\ref{cfe2}) is equivalent to
(\ref{gla}); i.e., the set of solutions of (\ref{cfe2}) with $\{r_{n}\}$
satisfying (\ref{fe1}) is identical with the set of solutions of (\ref{gla}).
\end{theorem}

\begin{proof}
(a) Let $\rho=\mu e^{i\theta}=\mu\cos\theta+i\mu\sin\theta$ and $\gamma
=\bar{\rho}$ be complex conjugate roots of both $P$ and $Q$ with $\sin
\theta\not =0$ since $\rho\not \in \mathbb{R}$. Recall from the proof of Lemma
\ref{fsor2} that
\[
P(u)=(u-\rho)P_{1}(u),\quad Q(u)=(u-\rho)Q_{1}(u)
\]

Applying the same argument to the polynomials $P_{1}$ and $Q_{1}$ using their
common root $\bar{\rho}$ yields%
\begin{align*}
P(u)  &  =(u-\rho)(u-\bar{\rho})P_{2}(u)=(u^{2}-(\rho+\bar{\rho})u+\rho
\bar{\rho})P_{2}(u),\\
Q(u)  &  =(u-\rho)(u-\bar{\rho})Q_{2}(u)=(u^{2}-(\rho+\bar{\rho})u+\rho
\bar{\rho})Q_{2}(u)
\end{align*}

where
\[
P_{2}(u)=u^{k-1}-\sum_{j=0}^{k-2}p_{j}^{\prime}u^{k-j-2},\quad Q_{2}%
(u)=\sum_{j=0}^{k-2}q_{j}^{\prime}u^{k-j-2}.
\]

Applying Lemma \ref{polyfac} to each of $P$ and $Q$ we obtain (\ref{p1j}) and
(\ref{q1j}) from (\ref{alj1}) since for every positive integer $m,$%
\[
\frac{\gamma^{m}-\rho^{m}}{\gamma-\rho}=\frac{-\mu^{m}(e^{i\theta
m}-e^{-i\theta m})}{-\mu(e^{i\theta}-e^{-i\theta})}=\mu^{m-1}\frac{\sin
m\theta}{\sin\theta}.
\]

(b) Eliminate $t_{n+1}$ and $t_{n}$ from (\ref{cfe2a}) using (\ref{cfe2b}) to
obtain%
\begin{gather*}
x_{n+1}-\rho x_{n}=\gamma(x_{n}-\rho x_{n-1})+r_{n+1},\quad\text{or:}\\
x_{n+1}-(\rho+\gamma)x_{n}+\rho\gamma x_{n-1}=r_{n+1}%
\end{gather*}

which is the same as (\ref{cfe2}). Now if $\{x_{n}\}$ is a solution of
(\ref{cfe2}) with a given sequence $\{r_{n}\}$ then by the preceding argument,
the sequence $\{x_{n}-\rho x_{n-1}\}$ satisfies (\ref{cfe2a}). Further, with
$t_{n}=x_{n}-\rho x_{n-1}$ it is clear that $\{x_{n}\}$ satisfies
(\ref{cfe2b}) so that the sequence of pairs $\{(t_{n},x_{n})\}$ is a solution
of the system of equations (\ref{cfe2a}) and (\ref{cfe2b}). Conversely, if
$\{(t_{n},x_{n})\}$ is a solution of the system then the above construction
shows that $\{x_{n}\}$ satisfies (\ref{cfe2}). Therefore, the same set of
solutions $\{x_{n}\}$ is obtained; i.e., the system is equivalent to the
second-order equation.

(c) The equivalence of the system of equations (\ref{fe1}) and (\ref{cfe2}) to
(\ref{gla}) is a consequence of Theorem 3.1 in \cite{bk2} (or Theorem 6 in
\cite{arx}).
\end{proof}

\pagebreak

\noindent\textbf{Remarks.}

\begin{enumerate}
\item Theorem \ref{cxrt} shows that the existence of a common \textit{complex}
(non-real) root $\rho$ for the polynomials $P,Q$ leads to a type-$(k-1,2)$
reduction, or factorization, of (\ref{gla}) over the real numbers. Over the
field of complex numbers $\mathbb{C}$, this reduction is equivalent to
repeated type-$(k,1)$ reductions as outlined in Lemmas \ref{fsor} and
\ref{fsor2}; see \cite{bk2} for the general background on reduction types.

\item The parameters $a_{j},b_{j}$, $j=k-1,k$ which affect $\rho$ but do not
appear in (\ref{p1j}) and (\ref{q1j}) are not free. They satisfy (\ref{cm-1})
and (\ref{cm}) in Lemma \ref{polyfac} and for the complex conjugate pair of
roots in Theorem \ref{cxrt} they take the forms
\begin{align}
a_{k}  &  =p_{k-2}^{\prime}\mu^{2},\quad a_{k-1}=p_{k-3}^{\prime}\mu
^{2}-2p_{k-2}^{\prime}\mu\cos\theta;\label{ak-1k}\\
b_{k}  &  =q_{k-2}^{\prime}\mu^{2},\quad b_{k-1}=q_{k-3}^{\prime}\mu
^{2}-2q_{k-2}^{\prime}\mu\cos\theta. \label{bk-1k}%
\end{align}

Here we assume that $p_{-1}^{\prime}=1$ and $q_{-1}^{\prime}=0$ when $k=2.$
\end{enumerate}

\section{Boundedness and periodicity}

In this section we use reduction of order and factorization methods of the
preceding section to prove the existence of oscillations in the real solutions
of certain difference equations of type (\ref{gla}). Convergence and global
attractivity issues regarding this equation are discussed in \cite{arx2} in at
a much more general level.

We quote the next result from the literature as a lemma; see \cite{SKy} or
Section 5.5 in \cite{bk2}. This result pertains to Equation (\ref{cfe2b})
whose solution may be written in the following way:%
\begin{equation}
x_{n}=\rho^{n}x_{0}+\sum_{j=1}^{n}\rho^{n-j}t_{j}. \label{cfs}%
\end{equation}

\begin{lemma}
\noindent\label{per}(periodicity, limit cycles, boundedness) Let $p$ be a
positive integer and\textit{\ let }$\rho\in\mathbb{C}$ with $\rho\not =0.$

(a) \textit{If for a given sequence }$\{t_{n}\}$\textit{ of complex numbers
Eq.(\ref{cfs}) has a solution }$\{x_{n}\}$\textit{ of period }$p$\textit{ then
}$\{t_{n}\}$\textit{ is periodic with period} $p.$

(b) \textit{Let }$\{t_{n}\}$\textit{ be a periodic sequence of complex numbers
with prime (or minimal) period }$p$ and assume that $\rho$ is not a $p$-th
root of unity; i.e., $\rho^{p}\not =1.$\textit{ If }$\{\tau_{0},\ldots
,\tau_{p-1}\}$\textit{ is one cycle of }$\{t_{n}\}$ \textit{and}%
\begin{equation}
\xi_{i}=\frac{1}{1-\rho{}\text{\/\hspace{0.01in}}^{p}{}}\sum_{j=0}^{p-1}%
\rho^{p-j-1}\tau_{(i+j)\operatorname{mod}p}\quad i=0,1,\ldots,p-1 \label{xi}%
\end{equation}
\textit{then the solution }$\{x_{n}\}$\textit{ of Eq.(\ref{cfs}) with }%
$x_{0}=\xi_{0}$\textit{ and }$t_{1}=\tau_{0}$ \textit{has prime period }%
$p$\textit{ and }$\{\xi_{0},\ldots,\xi_{p-1}\}$\textit{ is a cycle of
}$\{x_{n}\}.$

(c) If $|\rho|<1$ and $\{t_{n}\}$ is a sequence that converges to a p-cycle
then the sequence $\{x_{n}\}$ that is generated by \textit{(\ref{cfs})
converges to a }$p$-cycle\textit{. If }$\{\tau_{0},\ldots,\tau_{p-1}%
\}$\textit{ is one cycle of the limit of }$\{t_{n}\}$ then $\{\xi_{0}%
,\ldots,\xi_{p-1}\}$\textit{ is a cycle of the limit of }$\{x_{n}\}$ where
$\xi_{i}$ is defined by (\ref{xi}).

(d) If $|\rho|<1$ and $\{t_{n}\}$ is a bounded sequence with $|t_{n}|\leq M$
for all $n$ then the sequence $\{x_{n}\}$ that is generated by
\textit{(\ref{cfs}) is also bounded and there is a positive integer }$N$ such
that%
\[
|x_{n}|\leq|\rho|+\frac{M}{1-|\rho|}\quad\text{for all }n\geq N.
\]

\end{lemma}

Lemma \ref{per} and Theorem \ref{cxrt} imply the following result.

\begin{corollary}
\label{perb}Let $k\geq2$ in (\ref{gla}) and assume that the coefficients
$a_{i},b_{i}$ are real and $g_{n}:\mathbb{R}\rightarrow\mathbb{R}$ for
$n\geq0.$ If the polynomials $P,Q$ in Lemma \ref{fsor} have a common complex
root $\rho=\mu e^{i\theta}\not \in \mathbb{R}$ then the following statements
are true:

(a) If $\rho$ is not a $p$-th root of unity then for every periodic solution
of (\ref{fe1}) of prime period p (\ref{gla}) has a periodic solution of prime
period p that is given by (\ref{xi}).

(b) If modulus $|\rho|<1$ then for every limit cycle (attracting periodic
solution) of (\ref{fe1}) of period p (\ref{gla}) has a limit cycle of period
$p.$

(c) If modulus $|\rho|<1$ then for every bounded solution of (\ref{fe1}) the
corresponding solution of (\ref{gla}) is bounded. Hence, if every solution of
(\ref{fe1}) is bounded then every solution of (\ref{gla})\ is bounded.
\end{corollary}

\begin{proof}
We prove only (a) since the proofs of (b) and (c) use similar reasoning using
Lemma\textit{ }\ref{per}. Recall that the second-order cofactor equation
(\ref{cfe2}) in Theorem \ref{cxrt} is equivalent to the pair of first-order
cofactor equations (\ref{cfe2a}) and (\ref{cfe2b}). Let $\{r_{n}\}$ be a
solution of (\ref{fe1}) having prime period $p.$ If $\rho$ is not a $p$-th
root of unity then by Lemma\textit{ }\ref{per}(b) equation (\ref{cfe2a}) has a
solution $\{t_{n}\}$ in $\mathbb{C}$ with prime period $p$. Another
application of Lemma\textit{ }\ref{per} to equation (\ref{cfe2b}) shows that
the solution $\{x_{n}\}$ of (\ref{cfe2}) in $\mathbb{R}$ and hence, of
(\ref{gla}) also has prime period $p$.
\end{proof}

\medskip

It is worth pointing out that if $|\rho|\geq1$ then the periodic solution of
(\ref{gla}) in Corollary \ref{perb}(a) is not attracting even if the
corresponding solution $\{r_{n}\}$ of (\ref{fe1}) is attracting. Therefore,
such solutions may be difficult to identify numerically. Only when $|\rho|<1$
and the homogeneous part of the cofactor equation (\ref{cfe2}) fades away do
the solutions of the factor equation (\ref{fe1}) determine the asymptotic
behavior of solutions of (\ref{gla}).

In closing, we discuss the solutions of a third-order version of (\ref{gla}),
i.e., $k=2$ to illustrate the various aspects of the preceding results.
Consider the autonomous difference equation%
\begin{equation}
x_{n+1}=a_{0}x_{n}+a_{1}x_{n-1}+a_{2}x_{n-2}+g(x_{n}+b_{1}x_{n-1}+b_{2}%
x_{n-2}) \label{gla3}%
\end{equation}
where $a_{0},a_{1},a_{2},b_{1},b_{2}\in\mathbb{R}$ and $g:\mathbb{R}%
\rightarrow\mathbb{R}$. If $a_{2}=b_{2}=0$ then (\ref{gla3}) reduces to an
autonomous version of the second-order equation (\ref{sed1}). We assume here
that $b_{2}\not =0.$

The polynomial $Q$ of (\ref{gla3}) is the quadratic $u^{2}+b_{1}u+b_{2}$ whose
roots are complex if and only if $b_{1}^{2}<4b_{2}.$ These complex conjugate
roots are shared by the polynomial $P$ if and only if conditions (\ref{ak-1k})
hold. Since $k=2$ we calculate%
\begin{align*}
p_{0}  &  =\rho-a_{0},\quad p_{1}=\rho^{2}-a_{0}\rho-a_{1}\\
q_{0}  &  =1,\quad q_{1}=\rho+b_{1}\\
p_{0}^{\prime}  &  =\bar{\rho}+p_{0}=-b_{1}-a_{0},\quad q_{0}^{\prime}=1
\end{align*}

Note that $\rho+\bar{\rho}=-b_{1}$ and $\rho\bar{\rho}=b_{2},$

Thus conditions (\ref{ak-1k}) in this case are%
\begin{equation}
a_{1}=b_{1}(a_{0}+b_{1})-b_{2},\quad a_{2}=b_{2}(a_{0}+b_{1}). \label{ord}%
\end{equation}

We may alternatively obtain (\ref{ord}) using (\ref{cm-1}) and (\ref{cm}). If
$b_{1}^{2}<4b_{2}$ and conditions (\ref{ord}) hold then (\ref{gla3}) is
equivalent to the pair of equations%
\begin{align}
r_{n+1}  &  =(a_{0}+b_{1})r_{n}+g(r_{n}),\label{gla3f}\\
x_{n+1}  &  =-b_{1}\,x_{n}-b_{2}x_{n-1}+r_{n+1} \label{gla3cf}%
\end{align}
for $n\geq0$ where the intitial value of (\ref{gla3f}) is $r_{0}=x_{0}%
+b_{1}\,x_{-1}+b_{2}x_{-2}$ for a given triple of real initial values
$x_{0},x_{-1},x_{-2}$ for (\ref{gla3}).

Next, suppose that $g$ is a rational function of the following type%
\begin{equation}
g(u)=\frac{A}{u}+B+Cu,\quad A,B,C\in\mathbb{R},\ A\not =0. \label{gdef}%
\end{equation}

With this $g$ the difference equation is an example of a third-order rational
recursive equation. A second-order version of this equation is a rational
equation of type (\ref{ro2}) that is studied in \cite{dkmos}.

If $B=0$ and $C=-a_{0}-b_{1}$ then the factor equation (\ref{gla3f}) reduces
to%
\begin{equation}
r_{n+1}=\frac{A}{r_{n}}. \label{recip}%
\end{equation}

Every solution of (\ref{recip}) has period 2 with cycles $\{r_{0},A/r_{0}\}$
as long as $r_{0}\not =0.$ Since $\rho\not =\pm1$, corresponding to each
solution of (\ref{recip}) with period two, the solution of (\ref{gla3}) whose
triple of initial values $(x_{-2},x_{-1,}x_{0})$ is not on the plane
$u+b_{1}v+b_{2}w=0$ (so that $r_{0}\not =0$) has period 2. This plane which
passes through the origin is in fact the singularity (or forbidden) set of
(\ref{gla3}) in this case. The aforementioned periodic solutions of
(\ref{gla3}) have cycles $\{\xi_{0},\xi_{1}\}$ that we calculate in two stages
using (\ref{xi}). First, for (\ref{cfe2a}) with $\gamma=\bar{\rho}$ we
calculate the cycles $\{\tau_{0},\tau_{1}\}$ in $\mathbb{C}$ as%
\[
\tau_{0}=\frac{\bar{\rho}r_{0}+A/r_{0}}{1-\bar{\rho}^{2}},\quad\tau_{1}%
=\frac{r_{0}+\bar{\rho}A/r_{0}}{1-\bar{\rho}^{2}}.
\]

Next, using $\{\tau_{0},\tau_{1}\}$ in (\ref{xi}) we calculate the cycles
$\{\xi_{0},\xi_{1}\}$ for (\ref{gla3})%
\[
\xi_{0}=\frac{\rho\tau_{0}+\tau_{1}}{1-\rho^{2}}=\frac{\rho\bar{\rho}%
r_{0}+\rho A/r_{0}+r_{0}+\bar{\rho}A/r_{0}}{(1-\rho^{2})(1-\bar{\rho}^{2}%
)}=\frac{\rho\bar{\rho}r_{0}+(\rho+\bar{\rho})A/r_{0}+r_{0}}{1-(\rho^{2}%
+\bar{\rho}^{2})+\rho^{2}\bar{\rho}^{2}}%
\]

Since $\rho\bar{\rho}=b_{2}$ and $\rho+\bar{\rho}=-b_{1}$ it follows that
\[
\rho^{2}+\bar{\rho}^{2}=(\rho+\bar{\rho})^{2}-2\rho\bar{\rho}=b_{1}^{2}-2b_{2}%
\]
and thus,%
\[
\xi_{0}=\frac{b_{2}r_{0}-b_{1}A/r_{0}+r_{0}}{1-(b_{1}^{2}-2b_{2})+b_{2}^{2}%
}=\frac{r_{0}^{2}(b_{2}+1)-Ab_{1}}{r_{0}[(b_{2}+1)^{2}-b_{1}^{2}]}.
\]

As expected, $\xi_{0}\in$ $\mathbb{R}$. A\ similar calculation yields%
\[
\xi_{1}=\frac{A(b_{2}+1)-r_{0}^{2}b_{1}}{r_{0}[(b_{2}+1)^{2}-b_{1}^{2}]}.
\]

If $0<b_{2}<1$ then $|\rho|=\mu=\sqrt{b_{2}}<1$. In this case, every solution
of (\ref{gla3}) converges to a 2-cycle $\xi_{0}$ and $\xi_{1}$. These limit
cycles depend on $r_{0}$ and thus, on the initial values $x_{-2},x_{-1,}x_{0}$
in the sense that all initial points on the plane $u+b_{1}v+b_{2}w=r_{0}$
converge to the same limit cycle. However, if $b_{2}\geq1$ then other types of
solutions, including unbounded solutions are possible for (\ref{gla3}) that
are driven by the homogeneous part of (\ref{gla3cf}). To observe the 2-cycles
numerically it is necessary to use the initial values%
\[
x_{-2}=\xi_{0},\quad x_{-1}=\xi_{1},\quad x_{0}=r_{0}-b_{1}\xi_{1}-b_{2}%
\xi_{0}=\xi_{0}.
\]

Going in a different direction, if the function $g$ in (\ref{gdef}) has a
3-cycle then as is well-known, it has cycles of all possible lengths. In this
case, if $0<b_{2}<1$ then (\ref{gla3}) also has cycles of all possible
lengths. A set of parameter values that imply this situation is $A=1$,
$C=1-a_{0}-b_{1}$ and $B=-\sqrt{3}$; see \cite{dkmos}. In this case,
(\ref{gla3f}) is%
\[
r_{n+1}=\frac{1}{r_{n}}-\sqrt{3}+r_{n}%
\]
and its 3-cycle is found to be%
\[
\sigma_{0}=\frac{2}{\sqrt{3}}\left(  1+\cos\frac{\pi}{9}\right)  ,\quad
\sigma_{1}=g(\sigma_{0}),\quad\sigma_{2}=g(\sigma_{1}).
\]


\begin{thebibliography}{99}                                                                                               %


\bibitem {Clr}C.W. Clark, \textit{A delayed recruitment model of population
dynamics with an application to baleen whale populations}, J. Math. Biol. 3
(1976), pp.381-391.

\bibitem {dkmos}M. Dehghan, C.M. Kent, R. Mazrooei-Sebdani, N. Ortiz and H.
Sedaghat, \textit{Dynamics of rational difference equations containing
quadratic terms}, J. Difference Eq. Appl. 14 (2008), pp.191-208.

\bibitem {Elm}H.A. El-Morshedy, \textit{On the global attractivity and
oscillations in a class of second-order difference equations from
macroeconomics}, J. Difference Eq. Appl. 17 (2011), pp.1643-1650.

\bibitem {FG}M.E. Fisher and B.S. Goh, \textit{Stability results for delayed
recruitment in population dynamics}, J. Math. Biol. 19 (1984), pp. 147-156.

\bibitem {GLV}I. Gy\H{o}ri, G. Ladas and P.N. Vlahos, \textit{Global
attractivity in a delayed difference equation}, Nonlinear Analy. TMA, 17
(1991), pp.473-479.

\bibitem {Ham}Y. Hamaya, \textit{On the asymptotic behavior of solutions of
neuronic difference equations}, Proc. Int'l. Conf. Difference Eq., Special
Func. Appl., World Scientific, Singapore, 2007, pp.258-265.

\bibitem {Hic}J.R. Hicks, \textit{A Contribution to the Theory of the Trade
Cycle}, 2nd ed., Clarendon Press, Oxford, 1965.

\bibitem {KPS}G. Karakostas, C.G. Philos and Y.G. Sficas, \textit{The dynamics
of some discrete population models}, Nonlinear Anal., 17 (1991), pp.1069-1084.

\bibitem {KS}C.M. Kent and H. Sedaghat, \textit{Global stability and
boundedness in }$x_{n+1}=cx_{n}+f(x_{n}-x_{n-1})$, J. Difference Eq. Appl. 10
(2004), pp.1215-1227.

\bibitem {KL}V. Kocic and G. Ladas, \textit{Global Behavior of Nonlinear
Difference Equations of Higher Order with Applications}, Kluwer Academic,
Dordrecht, 1993.

\bibitem {LZ}S. Li and W. Zhang, \textit{Bifurcations in a second-order
difference equation from macroeconomics}, J. Difference Eq. Appl. 14 (2008), pp.91-104.

\bibitem {Liz}E. Liz, \textit{Stability of non-autonomous difference
equations: simple ideas leading to useful results}, J. Difference Eq. Appl. 17
(2011), pp.221-234.

\bibitem {Mem}R. Memarbashi, \textit{Sufficient conditions for the exponential
stability of nonautonomous difference equations}, Appl. Math. Lett. 21 (2008), pp.232-235.

\bibitem {Puu}T. Puu, \textit{Nonlinear Economic Dynamics}, 3rd. ed.,
Springer, New York, 1993.

\bibitem {Sam}P.A. Samuelson,\textit{ Interaction between the multiplier
analysis and the principle of acceleration}, Rev. Econ. Stat. 21 (1939), pp.75-78.

\bibitem {S97}H. Sedaghat, \textit{A class of nonlinear second-order
difference equations from macroeconomics}, Nonlinear Analy. TMA, 29 (1997), pp.593-603.

\bibitem {Sgeo}H. Sedaghat, \textit{Geometric stability conditions for higher
order difference equations}, J. Math. Anal. Appl. 224 (1998), pp.255-272.

\bibitem {bk1}H. Sedaghat, \textit{Nonlinear Difference Equations: Theory with
Applications to Social Science Models}, Kluwer Academic, Dordrecht, 2003.

\bibitem {Sed2}H. Sedaghat, \textit{On the Equation} $x_{n+1}=cx_{n}%
+f(x_{n}-x_{n-1})$, Fields Inst. Comm., 42 (2004), pp.323-326.

\bibitem {arx}H. Sedaghat, \textit{Form Symmetries and Reduction of Order in
Difference Equations}, arXiv:0907.3951, 2009.

\bibitem {SKy}H. Sedaghat, \textit{Periodic and chaotic behavior in a class of
second order difference equations}, Adv. Stud. Pure Math., 53 (2009), pp.311-318.

\bibitem {bk2}H. Sedaghat, \textit{Semiconjugate factorization and reduction
of order in difference Equations}, Chapman \& Hall/CRC Press, Boca Raton, 2011.

\bibitem {arx2}H. Sedaghat, \textit{Global attractivity in nonlinear higher
order difference equations in Banach algebras}, arXiv:1203.0227, 2012

\end{thebibliography}
\end{document}